\documentclass{amsart}
\usepackage{tikz}
\usepackage{float} 
\usepackage{amsmath,amssymb,lmodern,amsthm,amscd}

\usepackage{tcolorbox}

\newcommand{\K}{\mathbb{K}}
\newcommand{\pol}{\mathbb{K}[x_1,\ldots , x_n]}

\newcommand{\Hom}{\mathrm{Hom}}
\newcommand{\Ind}{\mathrm{Ind}}
\newcommand{\Res}{\mathrm{Res}}

\newcommand{\tprod}{\textstyle\prod}
\newcommand{\dd}[2][]{\frac{\partial #1}{\partial #2}} 
\newcommand{\bb}[1]{\mathbb{#1}}
\newcommand{\fr}[1]{\mathfrak{#1}}
\newcommand{\ca}[1]{\mathcal{#1}}
\newcommand{\ol}[1]{\overline{#1}}
\newcommand{\ra}{\rightarrow}
\newcommand{\ddx}{\frac{\partial}{\partial x}}

\newtheorem{lemma}{Lemma}

\newtheorem{prop}[lemma]{Proposition}
\newtheorem{cor}[lemma]{Corollary}

\newtheorem{thm}[lemma]{Theorem}

\newenvironment{example}[1][Example]{\begin{trivlist}
\item[\hskip \labelsep {\bfseries #1}]}{\end{trivlist}}

\begin{document}

\title[Simple $\mathfrak{sl}(V)$-modules which are free over an abelian subalgebra]{Simple $\mathfrak{sl}(V)$-modules which are free over an abelian subalgebra}

\author{Jonathan Nilsson}
\address{Mathematical Sciences, Chalmers University of Technology, Sweden}
\email{jonathn@chalmers.se}

\date{}

\begin{abstract}
\noindent Let $\mathfrak{p}$ be a parabolic subalgebra of $\mathfrak{sl}(V)$ of maximal dimension and let $\mathfrak{n} \subset \mathfrak{p}$ be the corresponding nilradical. In this paper we classify the set of $\mathfrak{sl}(V)$-modules whose restriction to $U(\mathfrak{n})$ is free of rank $1$. It turns out that isomorphism classes of such modules are parametrized by polynomials in $\dim V-1$ variables. We determine the submodule structure for these modules and we show that they generically are simple.
\end{abstract}

\maketitle

\section{Introduction}
Lie algebras and their representations appear throughout multiple areas of mathematics, and the elemental objects of representation theory are simple modules. Unfortunately, a complete classification of simple modules for a Lie algebra $\fr{g}$ is too broad a project, only for the Lie algebra $\fr{sl}_2$ does a version of such a classification exist, see~\cite{B,Maz1}. Nevertheless many classes of $\fr{g}$-modules are well studied. For example, when $\fr{g}$ is a simple finite-dimensional complex Lie algebra, all simple finite-dimensional modules were classified early, see \cite{Ca,Di}. More generally, simple highest weight modules (see \cite{Di,Hu,BGG}) and
simple weight modules with finite-dimensional weight spaces (see \cite{BL, Fu, Fe, Mat}) are also completely classified.

Several classes of non-weight modules have also been studied. These include Whittaker modules (see \cite{Kos,BM}), Gelfand-Zetlin modules (see \cite{DFO}), and various others (see for example ~\cite{FOS}).

Recently several authors have studied $\fr{g}$-modules whose restriction to certain $\fr{g}$-subalgebras are free. For example, when $\fr{g}$ is a simple complex finite-dimensional Lie algebra, the set of modules which are free of rank $1$ when restricted to the universal enveloping algebra of a Cartan subalgebra were classified in~\cite{N1,N2}. Corresponding and related results were also obtained for a multitude of other Lie algebras such as the Witt- and Virasoro-algebras, the Heisenberg-Virasoro algebra, Schrödinger algebras, and for basic Lie-super algebras, see 
\cite{CC,CG,CLNZ,CZ,CTZ,HCS,LZ,MP,N3,TZ1,TZ2} and references therein. A common theme for many of the modules in the papers listed above are that they are free when restricted to some commutative $\fr{g}$-subalgebra, often involving a Cartan subalgebra and central elements of $\fr{g}$. 

In the present paper we study $\fr{sl}(V)$-modules which are free over another maximal commutative subalgebra: the nilradical of a parabolic subalgebra of maximal dimension. A concrete example of such a module is given in the following result which is a restatement of Theorem~\ref{mainthm} in Section~\ref{sec:class}.

Let $\fr{g}=\fr{sl}_{n+1}$ and let $\fr{n}=\mathrm{span}(e_{1,n+1}, \ldots, e_{n,n+1})\subset \fr{g}$ (where as usual $e_{i,j}$ are the standard basis elements for $\fr{gl}_{n+1}$). Then $\fr{n}$ is the nilradical of the parabolic subalgebra corresponding to removing the simple root whose root space is spanned by $e_{n+1,n}$.
\begin{thm}
\label{introthm}
Fix a polynomial $p({\bf x})\in \pol$ and for $1\leq i,j \leq n$ define polynomials
 \[ p_{ij}:=x_i\frac{\partial p}{\partial x_j}+\delta_{ij}p({\bf 0})/n \quad \text{ and } \quad q_i:=-\frac{1}{x_i} \int_{0}^{x_i} \sum_{r=1}^n (p_{ii}^rp_{ri}+x_rp_{ii}^{ir}+p_{ii}^{i}p_{rr})dx_i,\]
 where upper indices indicate derivatives: $p^k=\tfrac{\partial p}{\partial x_k}$.

Then the following action equips the space $M(p)=\pol$ with an $\fr{sl}_{n+1}$-module structure:
\begin{align*}
e_{i,n+1}\cdot f&=x_if\\
h_i\cdot f&=fp_{ii} + x_i\tfrac{\partial f}{\partial x_i}\\
e_{ij}\cdot f&=fp_{ij}+x_i\tfrac{\partial f}{\partial x_j}\\
e_{n+1,i}\cdot f&=q_if-\sum_r(p_{ri}\tfrac{\partial f}{\partial x_r}+p_{rr}\tfrac{\partial f}{\partial x_i}+x_r\tfrac{\partial^2 f}{\partial x_i\partial x_r})
\end{align*}
for $f\in \pol$,  $1\leq i,j \leq n$, and $h_i:=e_{ii}-\frac{1}{n+1}I$.

Moreover, any $\fr{sl}_{n+1}$-module $M$ for which $\Res_{U(\fr{n})}^{U(\fr{sl}_{n+1})}M$ is free of rank $1$ is isomorphic to $M(p)$ for a unique $p\in \pol$.
\end{thm}
The easiest case is when $p$ is constant as in the following example:

\begin{example}Taking $n=2$ and $p=-\tfrac{3}{2}\lambda \in \K$ we obtain the following $\fr{sl}_3$-module structure on $\K[x_1,x_2]$:
\begin{align*}
e_{13} \cdot f&=x_1f  & e_{23} \cdot f&=x_2f\\
h_1 \cdot f &= x_1\tfrac{\partial f}{\partial x_1} &  h_2 \cdot f &= x_2\tfrac{\partial f}{\partial x_2}\\
e_{12} \cdot f&=x_1\tfrac{\partial f}{\partial x_2} & e_{21} \cdot f&=x_2\tfrac{\partial f}{\partial x_1}\\
e_{31}\cdot f&=\lambda\tfrac{\partial f}{\partial x_1}-d(\tfrac{\partial f}{\partial x_1})\\
e_{32}\cdot f&=\lambda\tfrac{\partial f}{\partial x_2}-d(\tfrac{\partial f}{\partial x_2})
\end{align*}
where we have written $d$ for the degree operator $x_1\tfrac{\partial}{\partial x_1}+x_2\tfrac{\partial}{\partial x_2}$. 
In this case our module $M(-\tfrac{3\lambda}{2})$ is actually parabolically induced: Let $\fr{p}=\mathrm{span}(h_1,h_2,e_{12},e_{21},e_{31},e_{32})$, and let $\bb{K}_{\lambda}$ be the $1$-dimensional $\fr{p}$-module where $h_1$ and $h_2$ both act by $\lambda$, and the other basis elements of $\fr{p}$ act trivially. Then $\Ind_{U(\fr{p})}^{U(\fr{sl}_3)} \K_{\lambda} = U(\fr{sl}_3) \otimes_{U(\fr{p})} \K_{\lambda} \simeq M(-\tfrac{3\lambda}{2})$. However, when $p$ is nonconstant $M(p)$ is not parabolically induced. 
\end{example}

The layout of this paper is as follows. In Section~\ref{sec:pre} we discuss parabolic subalgebras of $\fr{sl}(V)$ and their nilradicals, and we look at some general theory for modules free over subalgebras. In Section~\ref{sec:sl2} we focus on $\fr{g}=\fr{sl}_2$, in this case we get somewhat nicer formulas for our module structure. We determine the Jordan-Hölder components of the modules we construct, and in Section~\ref{sec:tensor} we give a Glebsch-Gordan style decomposition theorem for tensor products of $U(\fr{n})$-free modules and finite-dimensional modules.
In Section~\ref{sec:sln} we generalize most of these results to $\fr{sl}_{n+1}$. In Section~\ref{sec:class} we obtain the classification of $U(\fr{n})$-free modules of rank $1$ for $\fr{sl}(V)$ and in Section~\ref{sec:sub} we prove that our modules are irreducible in general, and determine the submodule structure for the exceptional cases.

\section{Preliminaries}
\label{sec:pre}
Denote the nonnegative integers by $\bb{N}$, and let $\K$ be an algebraically closed field of characteristic zero. 

\subsection{Modules which are free over a subalgebra}
We first discuss some general results relating the modules we study to previously known modules.

In this section let $\fr{g}$ be an arbitrary Lie algebra over $\K$ and let $\fr{a} \subset \fr{g}$ be a subalgebra.

First let us briefly recall that we have an adjunction between the functors $\Res_{U(\fr{a})}^{U(\fr{g})}: U(\fr{g})\text{-Mod} \ra U(\fr{a})\text{-Mod}$ and 
$\Hom_{U(\fr{a})}(U(\fr{g}),-): U(\fr{a})\text{-Mod} \ra U(\fr{g})\text{-Mod}$. In particular this means that for every $\fr{g}$-module $M$  and any $\fr{a}$-module $N$ we have a natural vector space isomorphism 
\[\Hom_{U(\fr{a})}(\Res_{U(\fr{a})}^{U(\fr{g})}M,N) \simeq \Hom_{U(\fr{g})}\big(M,\Hom_{U(\fr{a})}(U(\fr{g}),N)\big).\]
Here the $\fr{g}$-action on $\Hom_{U(\fr{a})}(U(\fr{g}),N)$ is given by $(x\cdot f)(y):=f(yx)$, and the correspondence above maps $\varphi \in \Hom_{U(\fr{a})}(\Res_{U(\fr{a})}^{U(\fr{g})}M,N)$ to $\ol{\varphi} \in \Hom_{U(\fr{g})}\big(M,\Hom_{U(\fr{a})}(U(\fr{g}),N)\big)$, where $\ol{\varphi}(m)\in \Hom_{U(\fr{a})}(U(\fr{g}),N)$ is defined by $\ol{\varphi}(m)(x):=\varphi(x\cdot m)$.

Now let $\fr{g}$ be a finite-dimensional complex simple Lie algebra, and let $\fr{a}\subset \fr{g}$ be a subalgebra. Let $M$ be a $\fr{g}$-module such that $\Res_{U(\fr{a})}^{U(\fr{g})}M$ is a free module of rank $1$. This just means that $M\simeq U(\fr{a})$ as an $\fr{a}$-module. 
As vector spaces we then have 
\[\Hom_{U(\fr{a})}(\Res_{U(\fr{a})}^{U(\fr{g})}M,N)\simeq\Hom_{U(\fr{a})}(U(\fr{a}),N)\simeq N,\] so the $\Res-\Hom$ adjunction above gives \[\dim \Hom_{U(\fr{g})}\big(M,\Hom_{U(\fr{a})}(U(\fr{g}),N)\big) = \dim N.\]

For example, 
take $\fr{a}=\mathrm{span}(z)$ for some $z\in \fr{g}$, and take and take $N=\K_{\alpha}$ to be the one-dimensional $\fr{a}$-module where $z$ acts by the scalar $\alpha$.
Then the space $\Hom_{U(\fr{a})}(U(\fr{a}),\K_{\alpha})=\Hom_{\K[z]}(\K[z],\K_{\alpha})$ is one-dimensional and spanned by the evaluation map $\varphi_{\alpha}$ where $\varphi_{\alpha}(f(z))=f(\alpha)$.
Corresponding to $\varphi_{\alpha}$ we get a $\fr{g}$-submodule $\mathrm{Ker}(\ol{\varphi}_{\alpha}) \subset M$. We can describe this kernel explicitly: Let $f(z)\in \K[z]=U(\fr{a})$. Then 
\[f\in\mathrm{Ker}(\ol{\varphi}_{\alpha}) \Leftrightarrow \ol{\varphi}_{\alpha}(f(z))=0 \Leftrightarrow \ol{\varphi}_{\alpha}(f(z))(x)=0 \; \forall x\in U(\fr{g})\]
\[\Leftrightarrow \varphi_{\alpha}(x \cdot f(z))=0 \; \forall x\in U(\fr{g}) \Leftrightarrow (x \cdot f)(\alpha)=0 \; \forall x\in U(\fr{g}).\]

As an example we may take $\fr{a}$ to be a Cartan-subalgebra of $\fr{g}$. Then the situation becomes as in the papers~\cite{N1,N2}, where we studied simple module structures on $U(\fr{h})$. In particular, if we take $\fr{g}=\fr{sl}_2$ and pick a basis $\{x,y,h\}$ satisfying $[h,x]=x$, $[h,y]=-y$ and $[x,y]=2h$, for any scalar $b$ we have a module structure on $M_b=U(\fr{h})$ in which \[h\cdot f(h)=hf(h), x\cdot f(h)=(h+b)f(h-1), y\cdot f(h)=-(h-b)f(h+1).\]
As above, $\mathrm{Ker}(\ol{\varphi}_{\alpha})$ is a submodule for each $\alpha$, and the conditions from~\cite{N1} for $M_b$ to be simple correspond precisely to our derived condition $\mathrm{Ker} \ol{\varphi}_{\alpha} = M_b$ as above.

If we stick with $\fr{sl}_2$, another option is to instead take $\fr{a}=\mathrm{span}(x)$ and study modules free over $U(\fr{a})$. This is what we do in Section~\ref{sec:sl2} below.

We can generalize these results to construct a new type of modules for $\fr{sl}_n$. Here instead of taking $\fr{a}$ as the Cartan subalgebra, we pick $\fr{a}$ as another abelian subalgebra of dimension $n-1$, namely the nilradical of a parabolic subalgebra of maximal dimension. This is discussed starting from Section~\ref{sec:sln}.

\subsection{Nilradicals of maximal parabolics}
We first describe the set of parabolic subalgebras of $\fr{sl}(V)$ of maximal dimension.
Given a proper nontrivial subspace $\Delta \subset V$ we define subalgebras of $\fr{sl}(V)$ as follows:
\[\fr{p}_{\Delta}:=\mathrm{Stab}(\Delta)=\{f\in \fr{sl}(V) \;|\; f(\Delta)\subset \Delta \},\]
\[\fr{n}_{\Delta}:=\{f\in \fr{sl}(V) \;|\; f(V)\subset \Delta \}.\]
We summarize some classical results on such subalgebras, see~\cite[Lemma 7.3.1]{Kob} for details.
\begin{lemma}
We have
\begin{enumerate}
	\item $\fr{p}_{\Delta}$ is a parabolic subalgebra of $\fr{sl}(V)$.
	\item $\fr{p}_\Delta$ is maximal with respect to inclusion: it is not contained in any other parabolic subalgebra.
	\item $\fr{n}_{\Delta}$ is the nilradical of $\fr{p}_\Delta$.
	\item $\fr{n}_\Delta$ is an abelian subalgebra.
	\item $(\fr{p}_\Delta)^\perp = \fr{n}_\Delta$ with respect to the Killing form on $\fr{sl}(V)$.
	\item $\fr{n}_\Delta$ is an ideal of $\fr{p}_\Delta$ and $\fr{p}_\Delta/\fr{n}_\Delta$ is semi-simple.
\end{enumerate}
\end{lemma}

\begin{lemma}
The following statements are equivalent:
\begin{enumerate}
	\item $\fr{p}_\Delta$ is a parabolic subalgebra of maximal dimension.
	\item $\dim \fr{n}_\Delta = \dim V -1$.
	\item $\dim \Delta=1$ or $\mathrm{codim}\: \Delta=1$.
\end{enumerate}
\end{lemma}
For any subspace $\Delta \subset V$, denote by $\fr{C}_{\Delta}$ the full subcategory $\fr{sl}(V)$-Mod consisting of modules which are free of rank $1$ when restricted to $U(\fr{n}_{\Delta})$. 

When $\Delta \subset V$ is a one-dimensional subspace, we may fix a basis $v_1 \ldots v_{n}$ of $V$ such that $v_{n}\in \Delta$. We write $\Delta^{\perp}$ for the subspace spanned by $v_1,\ldots, v_{n-1}$. This choice of basis lets us identify $\fr{sl}(V) = \fr{sl}_{n}$ which gives
\[\fr{p}_{\Delta} = 
\begin{bmatrix}
* &    *       & \cdots & * & 0 \\
* &   *       & \cdots & * & 0 \\
\vdots &   \vdots  & \ddots & \vdots & \vdots \\
* &   *       & \cdots & * & 0 \\
* &   *       & \cdots & * & * 
\end{bmatrix}
\qquad 
\fr{p}_{\Delta^\perp} = 
\begin{bmatrix}
* &   *       & \cdots & * & * \\
* &   *       & \cdots & * & * \\
\vdots &   \vdots  & \ddots & \vdots & \vdots \\
* &   *       & \cdots & * & * \\
0 &   0       & \cdots & 0 & * 
\end{bmatrix}
\]

\[\fr{n}_{\Delta} = 
\begin{bmatrix}
0 &   0       & \cdots & 0 & 0 \\
0 &   0       & \cdots & 0 & 0 \\
\vdots &   \vdots  & \ddots & \vdots & \vdots \\
0 &   0       & \cdots & 0 & 0 \\
* &   *       & \cdots & * & 0 
\end{bmatrix}
\qquad 
\fr{n}_{\Delta^\perp} = 
\begin{bmatrix}
0 &   0       & \cdots & 0 & * \\
0 &   0       & \cdots & 0 & * \\
\vdots &   \vdots  & \ddots & \vdots & \vdots \\
0 &   0       & \cdots & 0 & * \\
0 &   0       & \cdots & 0 & 0 
\end{bmatrix}.
\]

\begin{lemma}
There are equivalences of categories $\fr{C}_{\Delta} \simeq \fr{C}_{\Delta'}$ for any pairs of subspaces $\Delta$ and $\Delta'$ of either dimension or codimension $1$.
\end{lemma}
\begin{proof}
First assume that $\dim \Delta=\dim \Delta'$. Then we may pick an invertible $S$ that bijectively maps $\Delta$ to $\Delta'$. But then the  automorphism $\varphi:x\mapsto SxS^{-1}$ maps $\fr{n}_{\Delta}$ to $\fr{n}_{\Delta'}$. And therefore, if $M$ is a module free over $\fr{n}_{\Delta'}$, then the twisted module ${}^{\varphi}M$ (in which the action is $x\bullet m:=\varphi(x)\cdot m$) will be free over $\fr{n}_{\Delta}$. 
Similarly, we note that if $\dim \Delta=1$ the category $\fr{C}_{\Delta}$ is equivalent to $\fr{C}_{\Delta^{\perp}}$ by twisting by the outer isomorphism $x\mapsto-x^{T}$ (minus transpose). 
Finally, by combining the above statements we see that if $\dim \Delta = 1$ and $\mathrm{codim}\:\Delta'=1$ we have
 $\fr{C}_{\Delta} \simeq \fr{C}_{\Delta^{\perp}} \simeq \fr{C}_{\Delta'}$.
\end{proof}

Thus we shall restrict our focus to modules which are free of rank $1$ over the universal enveloping algebra of the fixed subalgebra $\fr{n}_{\Delta^\perp}$ of $\fr{sl}_n$ as described above. All other $\fr{sl}(V)$-modules free over the nilradical of a maximal-dimensional parabolic can be obtained from these by twisting by automorphisms.

\section{$\fr{sl}_2$-modules}
\label{sec:sl2}
We treat the case $\dim V=2$ separately, because we obtain more extensive results and nicer formulas in this setting.
We fix the standard basis $\{x,y,h\}$ for $\fr{sl}_2(\bb{K})$. These elements satisfy $[h,x]=2x,  [h,y]=-2y,\; [x,y]=h$.

When $\dim V=2$ the only parabolic subalgebras of $\fr{sl}(V)$ are Borel-subalgebras $\fr{p}=\fr{h}\oplus \fr{n}$, where the corresponding Cartan and nilradical subalgebras both are $1$-dimensional. 
By picking a basis $(v_1,v_2)$ for $V$ where $v_1\in \fr{h}$ and $v_2 \in \fr{n}$, we obtain an identification $\fr{sl}(V)=\fr{sl}_2$, $\fr{p}=\mathrm{span}(h,x)$, and $\fr{n}=\mathrm{span}(x)$. Since $U(\fr{n})=\K[x]$, our classification problem reduces to describing all possible $\fr{sl}_2$-module structures on $\K[x]$ in which we have $x\cdot f(x)=xf(x)$.

We start by defining some such modules.
\begin{prop}
\label{sl2def}
Fix a polynomial $p(x)\in \K[x]$ and define a second polynomial \[q(x):=-\frac{1}{2x}\int_{0}^{x}\big( p(t)p'(t)+tp''(t)\big) dt.\]

Then the following $\fr{sl}_2$-action defines an $\fr{sl}_2$-module structure on $\K[x]$: 
\begin{align*}
x\cdot f(x)&=xf(x),\\
h\cdot f(x)&=p(x)f(x)+2xf'(x),\\
y\cdot f(x)&=q(x)f(x)-p(x)f'(x)-xf''(x).\\
\end{align*}
We denote this module by $V(p)$.
\end{prop}
\begin{proof}
We verify that the above action respects the $\fr{sl}_2$ structure. We have
\begin{align*}
x \cdot y \cdot f -& y \cdot x \cdot f = x\cdot(qf-pf'-xf'')-y\cdot(xf)\\
=&x(qf-pf'-xf'')-\big(q(xf)-p(f+xf')-x(2f'+xf'')\big)\\
=&pf+2xf'=h\cdot f = [x,y] \cdot f,
\end{align*}
and
\begin{align*}
h \cdot x \cdot f -& x \cdot h \cdot f = h\cdot(xf)-x\cdot(pf+2xf')\\
=&p(xf)+2x(f+xf')-x(pf+2xf')=2xf=[h,x]\cdot f,
\end{align*}
and
\begin{align*}
h \cdot y \cdot f -& y \cdot h \cdot f= h\cdot(qf-pf'-xf'')-y\cdot(pf+2xf')\\
=&p(qf-pf'-xf'')+2x(q'f+qf'-p'f'-pf''-f''-xf''')\\
  &-q(pf+2xf')+p(p'f+pf'+2f'+2xf'')\\
  &+x(p''f+p'f'+p'f'+pf''+4f''+2xf''')\\
=&\big(2xq'+2q+pp'+xp''\big)f-2(qf-pf'-xf'')\\
=&\big(2\dd{x}(xq)+pp'+xp'')f-2(qf-pf'-xf'')\\
=&\bigg(pp'+xp''-\dd{x}\int_{0}^{x} p(t)p'(t)+tp''(t) dt \bigg)f-2(qf-pf'-xf'')\\
=&\big(pp'+xp''-pp'-xp''\big)f-2(qf-pf'-xf'')=-2y\cdot f=[h,y]\cdot f.
\end{align*}
\end{proof}

It turns out that the modules $V(p)$ defined above are pairwise non-isomorphic and exhaust all modules whose restriction to $\K[x]$ is free of rank $1$.
\begin{prop}
\label{uniprop}
$V(p)\simeq V(\ol{p})$ if and only if $p=\ol{p}$.
\end{prop}
\begin{proof}
Let $\varphi: V(p) \ra  V(\ol{p})$ be an isomorphism. Then $\varphi(f)=f\varphi(1)$ so $\varphi(1)$ is a nonzero constant.
The relation $\varphi(h\cdot f)=h\cdot \varphi(f)$ is equivalent to the condition $\varphi(1)(p-\ol{p})=0$ so $p=\ol{p}$.
\end{proof}

\begin{prop}
\label{exhaustprop}
Any $M\in \fr{sl}_2$-Mod such that $\Res_{\K[x]}^{U(\fr{sl}_2)}M$ is free of rank $1$ is isomorphic to $V(p)$ for some polynomial $p$.
\end{prop}
\begin{proof}
Let $M=\K[x]$ with a given $\fr{sl}_2$-module structure such that $x\cdot f(x)= xf(x)$. Define $p(x):=h\cdot 1$.
We claim that this implies that $h\cdot x^k= (p(x)+2k)x^{k}$. Indeed, it holds for $k=0$, and by induction we have
\[h\cdot x^{k+1}=h \cdot x \cdot x^k = x \cdot h \cdot x^k + [h,x]\cdot x^k= x(p(x)+2k)x^{k}+2x\cdot x^k\]
\[=(p(x)+2k)x^{k+1}+2x^{k+1}=(p(x)+2(k+1))x^{k+1}.\]
Note that $h\cdot x^k= (p(x)+2k)x^{k}$ for all $k$ can be written more compactly as $h\cdot f=pf+2xf'$ as in the definition of $V(p)$ above.

Next, we define $q(x):=y\cdot 1$ and claim that this implies that $y\cdot f=qf-pf'-xf''$ as in the above definition. This equality is equivalent to $y\cdot x^k=qx^k-kpx^{k-1}-k(k-1)x^{k-1}=(qx-kp-k(k-1))x^{k-1}$ for all $k$. The latter statement can again be proved by induction: it holds trivially for $k=0$ and we have
\[y\cdot x^{k+1} = y\cdot x  \cdot x^k= x\cdot y \cdot x^k + [y,x]\cdot x^k=x(y \cdot x^k) -h\cdot x^k\]
\[=x((qx-kp-k(k-1))x^{k-1})-(p(x)+2k)x^{k}=x((qx-kp-k(k-1))x^{k-1})-(p+2k)x^{k}\]
\[=(qx-kp-k(k-1)-p-2k)x^{k}=(qx-(k+1)p-(k+1)((k+1)-1))x^{(k+1)-1}.\]
This proves that the $h$- and $y$-action are completely determined by $p$ and $q$. It remains only to verify that $q$ is uniquely determined by $p$ as in the above definition. For this we expand the equality $[h,y] \cdot f-(h\cdot y\cdot f - y\cdot h \cdot f)=0$. Our previous considerations show that the left side expands as follows.
\begin{align*}
0&=[h,y] \cdot f-(h\cdot y\cdot f - y\cdot h \cdot f)=-2y\cdot f - h\cdot (qf-pf'-xf'') + y\cdot (pf+2xf')\\
&=-2(qf-pf'-xf'')-p(qf-pf'-xf'')-2x(q'f+qf'-p'f'-pf''-f''-xf^{(3)})\\
&\quad +q(pf+2xf')-p(p'f+pf'+2f'+2xf'')-x(p''f+2f'p'+pf''+4f''+2xf^{(3)})\\
&=-(2q+2xq'+pp'+xp'')f.
\end{align*}
This should hold for all $f$, which implies that $2q+2xq'+pp'+xp''=0$. This can be rewritten $\ddx(xq)=-\frac{1}{2}(pp'+xp'')$, which has the unique polynomial solution $q=-\frac{1}{2x} \int (pp'+xp'') dx$ (where the integration constant is forced to be zero). Thus $q$ is determined by $p$ and the module structure is just as in the above definition.
\end{proof}

Next we investigate the simplicity of the modules $V(p)$.
\begin{prop}
The module $V(p)$ is simple if and only if $p(0) \not\in -\bb{N}$. Otherwise $V(p(x))$ has length $2$ and we have a short exact sequence
\[0 \ra V\big(p(x)-2p(0)+2\big) \ra V(p(x)) \ra L(-p(0)) \ra 0,\] where $L(-p(0))$ is the simple highest weight module of highest weight $-p(0)\in \bb{N}$. 
\end{prop}
\begin{proof}
Let $S\subset V(p(x))$ be a proper nonzero submodule. We first claim that $S$ is a homogeneous ideal of $\K[x]$. This follows because $d:=\frac{1}{2}(h-p(x)) \in U(\fr{sl}_2)$ is the degree operator, which acts by $f\mapsto xf'(x)$, so by repeatedly acting by $(k-d)$ for different $k\in \bb{N}$ we can reduce any element $f$ to its lowest degree homogeneous component. Thus $S=x^k\K[x]$ for some $k>0$. Then $S\ni (q(x)-y)\cdot x^k=k(p(x)+(k-1))x^{k-1}$, so $x^k | k(p(x)+(k-1))x^{k-1}$ and $x|(p(x)+(k-1))$, which in turn means that $p(0)=1-k$. Thus if $p(0)\not\in -\bb{N}$, this is a contradiction so $V(p(x))$ is simple. On the other hand, if $p(0)=1-k\in -{\bb{N}}$ then $V(p(x))$ has a unique proper nontrivial submodule $x^k\K[x]$, so $V(p(x))$ has length $2$. Finally we analyze the quotient $V(p(x))/\langle x^k \rangle$ for $p(0)=1-k\in -\bb{N}$. This quotient is finite-dimensional and we have $x\cdot x^{k-1}=0$ and 
\[h\cdot x^{k-1}=(p(0)+2(k-1))x^{k-1}=(1-k+2(k-1))x^{k-1}=(k-1)x^{k-1},\] so $x^{k-1}$ is a highest weight vector, and the quotient is the simple highest weight module $L(k-1)$, and we recall that we had $k-1=-p(0)$. 

It remains only to verify that the submodule $x^k\K[x]$ is isomorphic to $V(p(x)-2p(0)+2)$. The submodule $x^k\K[x]$ is free of rank $1$ over $\K[x]$ so by Proposition~\ref{exhaustprop} we have $x^k\K[x]\simeq V(\ol{p}(x))$ for some polynomial $\ol{p}$. Any isomorphism $\varphi:x^k\K[x]\rightarrow V(\ol{p}(x))$ must be a multiple of $\varphi:V(\ol{p}(x)) \ra V(p(x))$ defined by $\varphi(f)=x^kf$ since both modules are free over $\K[x]$ and $\varphi$ needs to be bijective map between $V(\ol{p}(x))$ and the submodule $x^k\K[x] \subset V(p(x))$. Since $\varphi$ is an isomorphism we have \[x^kpf+2xkx^{k-1}+2x^{k+1}f'=h\cdot \varphi(f)= \varphi(h\cdot f)=x^k\ol{p}f+2x^{k+1},\]
from which it follows that $\ol{p}(x)=p(x)+2k=p(x)+2(1-p(0))$.
\end{proof}

Actually, the family of modules $V(p(x))$ includes the lowest weight Verma-modules as seen below. 
Take $p=\lambda \in \K$. Then $q=0$ and the action on the basis $\{x^k\}$ of the module $V(\lambda)$ is given by
\begin{align*}
x\cdot x^k&=x^{k+1},\\
h\cdot x^k&=(\lambda+2k)x^k,\\
y\cdot x^k&=-k(\lambda +(k-1))x^{k-1}.
\end{align*}
Thus each $x^k$ is a weight vector of weight $\lambda+2k$, and $V(\lambda)$ is a weight module of lowest weight $\lambda$. As in the proposition, $V(\lambda)$ is reducible precisely when $p(0)=\lambda \in -\bb{N}$. The quotient is the unique simple highest weight module of highest weight $-\lambda$.

\subsection{Tensor product decomposition}
\label{sec:tensor}
In this section we shall give a formula for decomposing $V(p(x))\otimes E$ when $E$ is a finite-dimensional and $V(p(x))$ is simple.

Let $L(k)$ be the unique simple $\fr{sl}_2$ module of dimension $k+1$. For natural numbers $k\geq m$ we then have
\[L(k)\otimes L(m) \simeq L(k+m) \oplus L(k+m-2) \oplus \cdots \oplus L(k-m),\]
which is known as the Clebsch-Gordan formula (see for example~\cite{Maz1}). 

Recall that $L(1)$ is isomorphic to the natural module; it has a basis $\{e_1,e_2\}$ on which $\fr{sl}_2$ acts by $e_{ij}\cdot e_k=\delta_{jk}e_i$. Then 
\[V(p)\otimes L(1)=\{(f,g):=f\otimes e_1 + g\otimes e_2 \: | \: f,g\in \K[x]\},\]
and using Proposition~\ref{sl2def} we see that the $\fr{sl}_2$ action on the tensor product is given by
\begin{align*}
x \cdot (f,g) &= \big(xf+g,xg\big),\\
h \cdot (f,g) &= \big((p+1)f+2xf',(p-1)g+2xg'\big),\\
y \cdot (f,g) &= \big(qf-pf'-xf'',qg-pg'-xg''+f\big).
\end{align*}

\begin{lemma}
\label{L1tensor}
For any polynomial $p$ we have
\[V(p)\otimes L(1) \simeq V(p-1)\oplus V(p+1).\]
\end{lemma}
\begin{proof}
The action of $\K[x]$ on the tensor product can be written $r(x)\cdot (f,g)=(rf+r'g,rg)$, so any submodule isomorphic to $\K[x]$ can be generated by a single element $(f,g)$. By taking $(f,g):=(\tfrac{1}{2x}(p-p(0)),1)$ we get a $\K[x]$-submodule as follows:
Define $\varphi: \K[x] \ra V(p)\otimes L(1)$ by \[\varphi(f):= (\tfrac{1}{2x}(p-p(0))f+f',f).\]
Then \[\ol{V}:=\mathrm{Im}\; \varphi=\{\varphi(f)=(\tfrac{1}{2x}(p-p(0))f+f',f) \in V(p)\otimes L(1) \: | \: f\in \K[x]\}\]
is a $\K[X]$-submodule in which $x\cdot \varphi(f)=\varphi(xf)$.
We claim that $\ol{V}$ in fact is an $\fr{sl}_2$-submodule. Since $\varphi(f)=f \cdot \varphi(1)$ it suffices to verify that $h\cdot \varphi(1)\in \ol{V}$ and $h\cdot \varphi(1)\in \ol{V}$. We have
\begin{align*}
h\cdot \varphi(1)&=h\cdot (\tfrac{1}{2x}(p-p(0)),1)=\big((p+1)\tfrac{1}{2x}(p-p(0))+2x\tfrac{1}{2x}(p-p(0))',p-1\big)\\
 &=\big((p+1)\tfrac{1}{2x}(p-p(0))-\tfrac{1}{x}(p-p(0))+p',p-1\big)\\
 &=\big((p-1)\tfrac{1}{2x}(p-p(0))+p',p-1\big) = \varphi(p-1).
\end{align*}
Next we claim that $y\cdot \varphi(1) = \varphi(q+\tfrac{1}{2x}(p-p(0)))$. To see this first note that
 \[-2xq'=(-2xq)'+2q = \dd{x}\big(\int_{0}^x (pp'+xp'')dt\big)+2q=pp'+xp''-p(0)p'(0)+2q.\]
We now calculate
\begin{align*}
(y\cdot \varphi(1))&-\varphi(q+\tfrac{1}{2x}(p-p(0)))=\\
 =&\Big((p-p(0))\big(\tfrac{q}{2x}+\tfrac{p}{2x^2}-\tfrac{1}{x^2}\big)+\tfrac{p'}{2x}(2-p)-\tfrac{p''}{2} , q+\tfrac{1}{2x}(p-p(0))\Big)\\
 &-\Big(\tfrac{1}{2x}(p-p(0))\big(q+\tfrac{1}{2x}(p-p(0))+q'-\tfrac{1}{2x^2}(p-p(0))+\tfrac{p'}{2x}\big) , q+\tfrac{1}{2x}(p-p(0))\Big).
\end{align*}
In this difference the second component is indeed zero and in the first component we obtain the following after multiplying by $2x$:
\begin{equation*}
-2xq'+(p-p(0))\big(\tfrac{p}{x}-\tfrac{1}{x}-\tfrac{1}{2x}(p-p(0))\big)-pp'+p'-xp'',
\end{equation*}
and we need to show that this is zero too. Inserting our expression above for $-2xq'$  and multiplying again by $x$ we get
\begin{equation*}
2xq-xp(0)p'(0)+(p-p(0))\big(\tfrac{p+p(0)}{2}-1\big)+xp',
\end{equation*}
which is zero when $x=0$. The derivative is
\begin{equation*}
-\big(pp'+xp''-p(0)p'(x)\big)-p(0)p'(0)+p'\big(\tfrac{p+p(0)}{2}-1\big)+(p-p(0))\tfrac{p'}{2}+xp''+p'=0.
\end{equation*}
Hence we have shown that $y \cdot \varphi(1) = \varphi(q+\tfrac{1}{2x}(p-p(0))) \in \ol{V}$, and it follows that $\ol{V}$ is an $\fr{sl}_2$-submodule.

Next we define $\psi:\K[x] \ra V(p)\otimes L(1)$ by
\[\psi(f)=(\tfrac{1}{2}(p+p(0))f+xf',xf)\] so that
\[\tilde{V}:=\mathrm{Im}\: \psi=\{\psi(f)=(\tfrac{1}{2}(p+p(0))f+xf',xf) \in V(p)\otimes L(1) \: | \: f\in \K[x]\}.\]
We claim that $\tilde{V}$ is a $\fr{sl}_2$ submodule complementary to $\ol{V}$.
Verification of this is analogous to the calculations above and we omit it here. We note however that
$x\cdot \psi(f)=\psi(xf)$, $h\cdot \psi(1)=\psi(p+1)$ and $y \cdot \psi(1) = \psi(q-\tfrac{1}{2x}(p-p(0)))$.

By Propositions~\ref{uniprop} and~\ref{exhaustprop} it follows from the facts that $h\cdot \varphi(1)=\varphi(p-1)$ and $h\cdot \psi(1)=\psi(p+1)$ that $\varphi$ is an isomorphism $V(p-1) \ra \ol{V}$ and that $\psi$ is an isomorphism $V(p+1) \ra \tilde{V}$.

Finally we verify that $V(p)=\ol{V}\oplus \tilde{V}$. 
Since
\[\varphi(xf)-\psi(f)=(\tfrac{1}{2}(p-p(0))f+(xf)',xf)-(\tfrac{1}{2}(p+p(0))f+xf',xf)=((1-p(0))f,0),\]
and since $p(0)\neq 1$ we have $(\K[x],0) \subset \ol{V}+\tilde{V}$ and then clearly also $(0,\K[x])\subset \ol{V}+\tilde{V}$ since we may form $\varphi(g)-(\tfrac{1}{2x}(p-p(0))g+g',0)=(0,g)$. Thus $V(p)\otimes L(1)=\ol{V}+\tilde{V}$. Next, assume that $\varphi(g)=\psi(f)$, then $g=xf$ and
$(\tfrac{1}{2}(p-p(0))f+(xf)',xf)=(\tfrac{1}{2}(p+p(0))f+xf',xf)$
so $\tfrac{1}{2}(p-p(0))f+xf'+f=\tfrac{1}{2}(p+p(0))f+xf'$ implying $(1-p(0))f=0$. Since $p(0)\neq 1$ we get $f=g=0$ showing that $\ol{V} \cap \tilde{V}=\{0\}$. Thus $V(p)\otimes L(1)=\ol{V}\oplus \tilde{V}$.
\end{proof}

Since each finite-dimensional module $E$ is a direct sum of modules $L(k)$, the following proposition determines the tensor product decomposition of $V(p)\otimes E$ completely.
\begin{prop}
\label{tensorprop}
We have
\[V\big(p(x)\big)\otimes L(k) = \bigoplus_{i=0}^k V\big(p(x)+k-2i\big).\]
\end{prop}
\begin{proof}
We proceed by induction. The statement holds trivially for $k=0$, and also for $k=1$ by Lemma~\ref{L1tensor}. Using the inductive assumption and Lemma~\ref{L1tensor} we find that
\begin{align*}
\big(V(p)&\otimes L(k)\big)\otimes L(1) = \bigoplus_{i=0}^k \big(V(p+k-2i)\otimes L(1)\big)\\
&=\bigoplus_{i=0}^k \big(V(p+(k+1)-2i)\oplus V(p+(k-1)-2i)\big)\\
&=\bigoplus_{i=0}^{k} V(p+(k+1)-2i) \oplus V(p-k-1) \oplus \bigoplus_{i=0}^{k-1} V(p+(k-1)-2i)\\
&=\bigoplus_{i=0}^{k+1} V(p+(k+1)-2i) \oplus \bigoplus_{i=0}^{k-1} V(p+(k-1)-2i).
\end{align*}
But on the other hand we can use the Clebsch-Gordan formula to obtain
\begin{align*}
\big(V(p)&\otimes L(k)\big)\otimes L(1) = V(p)\otimes \big(L(k)\otimes L(1)\big)=V(p)\otimes \big(L(k+1)\oplus L(k-1)\big)\\
&=\big(V(p)\otimes L(k+1)\big) \oplus  \bigoplus_{i=0}^{k-1} V\big(p+(k-1)-2i\big).
\end{align*}
By cancelling the isomorphic summands $\bigoplus_{i=0}^{k-1} V\big(p+(k-1)-2i\big)$ in the two above expressions we finally get
\[V(p)\otimes L(k+1) =  \bigoplus_{i=0}^{k+1} V\big(p(x)+(k+1)-2i\big),\]
and the statement of the proposition follows by induction.
\end{proof}

\section{$\fr{sl}_{n+1}$-modules}
In this section we generalize most of the results of the previous section from $\fr{sl}_2$ to $\fr{sl}_{n+1}$.
\label{sec:sln}
\subsection{Preliminaries}
For $1\leq i,j\leq n+1$, let $e_{ij}$  be the standard basis for $\fr{gl}_{n+1}$, and recall that these satisfy $[e_{ij},e_{kl}]=\delta_{jk}e_{il}-\delta_{li}e_{jk}$. Let $\fr{n} \subset \fr{sl}_{n+1}$ be the subalgebra with basis $\{e_{i,n+1}\; | \; 1\leq i \leq n \}$. Since $\fr{n}$ is abelian we have $U(\fr{n})\simeq \K[x_1,\ldots , x_n]$.

For $f\in \pol$ write $f^i:=\dd{x_i}f$ for the partial derivative, and define degree operators
\[d,d_i:\pol \ra \pol \text{ by } d_i(f)=x_if^i \text{ and } d(f)=\sum_{i=1}^n x_if^i.\]
Note that $d_i$ and $d$ are $\K$-linear derivations of $\pol$. We also note that $d$ is invertible on the space of polynomials with zero constant term. We define maps $d',d_i':\pol \ra \pol$ by
\[d'(x_1^{a_1} \cdots x_n^{a_n})=\frac{1}{\sum a_i}x_1^{a_1} \cdots x_n^{a_n} \text{ when } \sum a_i>0, \text{ and } d'|_{\K}:=\mathrm{id}.\]
and
\[d_i'(x_1^{a_1} \cdots x_n^{a_n})=\frac{1}{a_i}x_1^{a_1} \cdots x_n^{a_n} \text{ when } a_i>0, \text{ and } d_i'(x_1^{a_1} \cdots x_n^{a_n})=x_1^{a_1} \cdots x_n^{a_n} \text{ when } a_i=0.\]
Then the following lemma is easy to verify.
\begin{lemma}
\label{dlemma}
For all $1\leq i,j\leq n$ we have
\begin{align*}
d \circ d'(f)  &= f-f(0,\ldots, 0)= d' \circ d(f),\\
d_i \circ d_i'(f)  &= f-f(x_1,\ldots,0_i,\ldots,x_n)= d_i' \circ d_i(f),\\
d(x_i\dd{x_j}f)&=x_i\dd{x_j}d(f),\\
[\dd{x_i},d]&=\dd{x_i}.\\
\end{align*} 
\end{lemma}

For $1\leq i \leq n$ define $x_i:=e_{i,n+1}$. Let $I:=\sum_{i=1}^{n+1}e_{ii}$ and for $1\leq i \leq n+1$ define $h_i:=e_{ii}-\frac{1}{n+1}I$ and $\ol{h}=\sum_{i=1}^{n}h_i$. Note that $h_{n+1}=-\sum_{i=1}^n h_i=-\ol{h}$.

We fix the basis $\{e_{ij}\;|\;1\leq i,j \leq n+1; i\neq j\}\cup\{h_1,\ldots ,h_n\}$ for $\fr{sl}_{n+1}$ and note that for $1\leq i,j\leq n$ and $i\neq j$ we have 
\[[h_i,x_j]=\delta_{ij}x_j \text{ and } [x_i,e_{n+1,i}]=\ol{h}+h_i.\] The following lemma tells us how to commute elements of $\fr{sl}_{n+1}$ with the $x_i$.

\begin{lemma}
\label{Urel}
The following relations hold in $U(\fr{sl}_{n+1})$ for all $1 \leq i,j,k \leq n$ with $i\neq j$ and for all $m\in \bb{N}$.
\begin{align}
x_ix_k^m&=x_k^mx_i \label{eq:1}\\
h_ix_k^m&=x_k^mh_i +\delta_{ik}mx_k^m \label{eq:2}\\
e_{ij}x_k^m&=x_k^me_{ij}+\delta_{kj}mx_ix_k^{m-1} \label{eq:3}\\
e_{n+1,i}x_j^m&=x_j^me_{n+1,i} - mx_j^{m-1}e_{ji} \label{eq:4}\\
e_{n+1,i}x_i^m&=x_i^me_{n+1,i} - mx_i^{m-1}(\ol{h}+h_i)-m(m-1)x_i^{m-1} \label{eq:5}
\end{align}
\end{lemma}
\begin{proof}
These relations can be easily proved by induction on $m$. We verify only Equation~\eqref{eq:5}. It holds trivially for $m=0$ and also for $m=1$. Assuming it holds for a fixed $m$, we have
\begin{align*}
e_{n+1,i}x_i^{m+1}&=(x_i^me_{n+1,i} - mx_i^{m-1}(\ol{h}+h_i)-m(m-1)x_i^{m-1})x_i\\
	&=x_i^{m}\big(x_ie_{n+1,i}-(\ol{h}+h_i)\big) - mx_i^{m-1}\big(x_i(\ol{h}+h_i)+2x_i\big)-m(m-1)x_i^{m}\\
	&=x_i^{m+1}e_{n+1,i}-(m+1)x_i^m(\ol{h}+h_i)-m(m+1)x_i^{m},
\end{align*}
where in the second equality we used Equations~\eqref{eq:2} and~\eqref{eq:5} for $m=1$. Thus Equation~\eqref{eq:5} holds for all $m\in\bb{N}$ by induction.
\end{proof}

\begin{cor}
\label{Ucor}
Let $f\in \pol = U(\fr{n})$. The following relations hold in $U(\fr{sl}_{n+1})$ for all $1 \leq i,j \leq n$ with $i\neq j$.
\begin{align}
x_if&=fx_i \label{cor:1}\\
h_if&=fh_i + x_if^i \label{cor:2}\\
e_{ij}f&=fe_{ij}+x_if^j \label{cor:3}\\
e_{n+1,i}f&=fe_{n+1,i}-\sum_{k\neq i}f^ke_{ki}-f^i(\ol{h}+h_i)-d(f^i) \label{cor:4}
\end{align}
\end{cor}
\begin{proof}
Since the above formulas are linear in $f$, it suffices to prove them when $f$ is a monomial. We do by using Lemma~\ref{Urel} repeatedly.  Equations~\eqref{cor:1}-\eqref{cor:2} follows easily from Equations~\eqref{eq:1}-\eqref{eq:2}. For \eqref{cor:3} we take $f=\prod x_k^{a_k}$ and compute
\begin{align*}
e_{ij}f&=(e_{ij}x_j^{a_j})\tprod_{k\neq j}x_k^{a_k}=(x_j^{a_j}e_{ij}+a_jx_ix_j^{a_j-1})\tprod_{k\neq j}x_k^{a_k}\\
&= x_j^{a_j}\tprod_{k\neq j}x_k^{a_k}e_{ij}+a_jx_ix_j^{a_j-1}\tprod_{k\neq j}x_k^{a_k}=fe_{ij}+x_if^j
\end{align*}
where we used that $e_{ij}$ commutes with $x_k$ for $k\neq j$. For \eqref{cor:4} we instead proceed by induction. The equation holds for $f=1$ and assuming that the equation holds for a fixed monomial $f$ it suffices to prove that it holds when we replace $f$ by $x_jf$. We divide this into cases: for $j\neq i$ we have

\begin{align*}
(x_jf)e_{n+1,i}&-\sum_{k\neq i}(x_jf)^ke_{ki}-(x_jf)^i(\ol{h}+h_i)-d((x_jf)^i)\\
     &=x_jfe_{n+1,i}-\big(x_j\sum_{k\neq i}f^ke_{ki}+fe_{ji}\big)-x_jf^i(\ol{h}+h_i)-x_j(d(f^i)+f^i)\\
     &=x_j\big(fe_{n+1,i}- \sum_{k\neq i}f^ke_{ki}- f^i(\ol{h}+h_i) -d(f^i) \big) -fe_{ji}-x_jf^i\\
     &=x_je_{n+1,i}f-e_{ji}f=(x_je_{n+1,i}-e_{ji})f=e_{n+1,i}(x_jf)
\end{align*}
And similarly, if we instead take $j=i$ we have
\begin{align*}
(x_if)e_{n+1,i}&-\sum_{k\neq i}(x_if)^ke_{ki}-(x_if)^i(\ol{h}+h_i)-d((x_if)^i)\\
     &=x_ife_{n+1,i}-x_i\sum_{k\neq i}f^ke_{ki}-(x_if^i+f)(\ol{h}+h_i)-x_i(d(f^i)+f^i)-d(f)\\
     &=x_i\big(fe_{n+1,i}- \sum_{k\neq i}f^ke_{ki} -f^i(\ol{h}+h_i) d(f^i) \big) -f(\ol{h}+h_i)-x_if^i-d(f)\\
     &=x_ie_{n+1,i}f-(\ol{h}+h_i)f=(x_ie_{n+1,i}-(\ol{h}+h_i))f=e_{n+1,i}(x_if)
\end{align*}
Therefore Equation~\eqref{cor:4} holds by induction.
\end{proof}

\subsection{Classification}
\label{sec:class}
In this section we shall study all possible $\fr{sl}_{n+1}$-modules $M$ such that $\Res^{U(\fr{sl}_{n+1})}_{U(\fr{n})}M$ is free of rank $1$. 

Let $M$ be a given $\fr{sl}_{n+1}$-modules $M$ such that $\Res^{U(\fr{sl}_{n+1})}_{U(\fr{n})}$ is free of rank $1$. 
Without loss of generality we may assume that $M=\pol$ as a vector space and that $e_{i,n+1}\cdot f=x_if$ for $f\in M$.
For $1\leq i,j\leq n$ and $i\neq j$, define $p_{ij}:=e_{ij}\cdot 1$, $p_{ii}:=h_{i}\cdot 1$ and $q_{i}:=e_{n+1,i}\cdot 1$. Also let $\ol{p}:=\sum_{i=1}^n p_{ii}$ and $p:=d'(\ol{p})$.
\begin{lemma}
\label{act}
The elements $p_{ij},q_i\in \pol$ uniquely determines the module structure on $M$. Explicitly, for $i\neq j$ we have
\begin{align*}
x_i\cdot f&=x_if \\
h_i\cdot f&=p_{ii}f + x_if^i \\
e_{ij}\cdot f&=p_{ij}f+x_if^j \\
e_{n+1,i}\cdot f&=q_if-\sum_{k=1}^np_{ki}f^k-\ol{p}f^i-d(f^i)\\
&=q_if-\sum_r(p_{ri}f^r+p_{rr}f^i+x_rf^{ir})
\end{align*}
where $\ol{p}=\sum_{i=1}^n p_{ii}$.
\end{lemma}
\begin{proof}
This immediately follows by acting on $1\in \pol$ in both sides of each equation of Corollary~\ref{Ucor}.
\end{proof}

Next we determine what relations are required between polynomials $p_{ij}$ and $q_i$ in order that the action in Lemma~\ref{act} should define an $\fr{sl}_{n+1}$-module structure on $\pol$.

First, let $\fr{p}\subset \fr{sl}_{n+1}$ be the parabolic subalgebra spanned by $\{h_1,\ldots,h_n\}$ and all $e_{i,j}$ for $1\leq i \leq n$ and $1\leq j \leq n+1$ where $i\neq j$. 

First we note that 
\begin{align*}
h_i \cdot h_j \cdot f -& h_j \cdot h_i \cdot f=h_i(p_{jj}f+x_jf^j)-h_j(p_{ii}f+x_if^i)\\
=&\quad p_{ii}(p_{jj}f+x_jf^j)+x_i(p_{jj}^if+p_{jj}f^i+x_jf^{ji})\\
&-p_{jj}(p_{ii}f+x_if^i)-x_j(p_{ii}^jf+p_{ii}f^j+x_if^{ij})\\
=&(x_ip_{jj}^i-x_jp_{ii}^j)f
\end{align*}
so the conditions that $h_i \cdot h_j \cdot f - h_j \cdot h_i \cdot f=[h_i,h_j]\cdot f = 0$ for all $f$ reduces to the conditions
\begin{equation}
\label{pii}
x_ip_{jj}^i=x_jp_{ii}^j.
\end{equation}
Summing over $j$ we obtain
$x_i\ol{p}^i=d(p_{ii})$ and applying $d'$ we have
\[d'(x_i\ol{p}^i)=d'\circ d(p_{ii})=p_{ii}-p_{ii}(0,\ldots,0)\]
 Thus each $p_{ii}$ is determined up to addition of a constant by $\ol{p}$.
We conclude that \[p_{ii}=d'(x_i\ol{p}^i)+c_i\]
for some constants $c_i$.

Next, for $1\leq i,j,k \leq n$ with $i\neq j$ we have 
\begin{align*}
h_k \cdot e_{ij} \cdot f -& e_{ij} \cdot h_k \cdot  f\\
=&h_k(p_{ij}f+x_if^j)-e_{ij}(p_{kk}f+x_kf^k)\\
=&\quad \: p_k(p_{ij}f+x_if^j)+x_k(p_{ij}^kf+p_{ij}f^k+x_if^{jk}+\delta{ik}f^j)\\
 &-p_{ij}(p_{kk}f+x_kf^k)-x_i(p_{kk}^jf+p_{kk}f^j+x_kf^{kj}+\delta_{kj}f^k)\\
=&(x_kp_{ij}^k-x_ip_{kk}^j)f+x_k\delta_{ik}f^j-x_i\delta_{kj}f^k.
\end{align*}
On the other hand, \[[h_k,e_{ij}]\cdot f = (\delta_{ki}e_{kj}-\delta_{kj}e_{ik}) \cdot f 
= (\delta_{ki}-\delta_{kj})e_{ij}\cdot f =(\delta_{ki}-\delta_{kj})(p_{ij}f+x_if^j)\]
So the conditions $[h_k,e_{ij}]\cdot f=h_k \cdot e_{ij} \cdot f - e_{ij} \cdot h_k \cdot f$ translates to the conditions
\begin{equation}
\label{pij}
x_kp_{ij}^k-x_ip_{kk}^j=(\delta_{ki}-\delta_{kj})p_{ij} \quad 1\leq i,j,k\leq n.
\end{equation}
Note that by Equation~\eqref{pii}, the above equality also holds when $i=j$.

Now for $1\leq i,j,k,l \leq n$ where $i\neq j$ and $k\neq l$ we have
\begin{align*}
e_{ij} \cdot e_{kl} &\cdot f -  e_{kl}\cdot e_{ij} \cdot f\\
  =&e_{ij}\cdot (p_{kl}f+x_kf^l)-e_{kl}\cdot (p_{ij}f+x_if^j)\\
  =& \quad \; p_{ij}(p_{kl}f+x_kf^l)+x_i(p_{kl}^jf+p_{kl}f^j+\delta_{kj}f^l+x_kf^{lj})\\
   &-p_{kl}(p_{ij}f+x_if^j)-x_k(p_{ij}^lf+p_{ij}f^l+\delta_{il}f^j+x_if^{jl})\\
  =&(x_ip_{kl}^j-x_kp_{ij}^l)f+\delta_{kj}x_if^l-\delta_{il}x_kf^j
\end{align*}
while \[[e_{ij},e_{kl}]\cdot f = \delta_{jk}e_{il}\cdot f-\delta_{il}e_{kj}\cdot f = \delta_{jk}(p_{il}f+x_if^l)-\delta_{il}(p_{kj}+x_kf^j),\]
So the condition $[e_{ij},e_{kl}]\cdot f=e_{ij}\cdot e_{kl}\cdot f-e_{kl}\cdot e_{ij}\cdot f$ for all $f$ is equivalent to
\begin{equation}
\label{pij3}
x_ip_{kl}^j-x_kp_{ij}^l=\delta_{kj}p_{il}-\delta_{il}p_{kj}
\end{equation}

\begin{lemma}
\label{qlemma}
We have
\[q_i=-\frac{1}{x_i}\int \sum_{r=1}^n (p_{ii}^rp_{ri}+x_rp_{ii}^{ir}+p_{ii}^{i}p_{rr}) dx_i=-\frac{1}{x_i}d_i'\big(x_i\sum_{r=1}^n (p_{ii}^rp_{ri}+x_rp_{ii}^{ir}+p_{ii}^{i}p_{rr})\big)\]
\end{lemma}
\begin{proof}
Since $[h_k,e_{n+1,i}]\cdot f=h_k \cdot e_{n+1,i} \cdot f-e_{n+1,i} \cdot h_k \cdot f$ for all $f$, we have
\begin{align*}
0=&h_k \cdot e_{n+1,i} \cdot f-e_{n+1,i} \cdot h_k \cdot f -[h_k,e_{n+1,i}]\cdot f=\\
=& \; h_k\cdot \big(q_if-\sum_r(p_{ri}f^r+p_{rr}f^i+x_rf^{ir})\big)\\
 &-e_{n+1,i}\cdot \big(p_{kk}f+x_kf^{k}\big) + \delta_{ki}e_{n+1,i} \cdot f\\
 =&p_{kk}\big(q_if-\sum_r(p_{ri}f^r+p_{rr}f^i+x_rf^{ir})\big)\\
 &+ x_k\big(q_i^kf+q_if^k-\sum_r(p_{ri}^kf^r+p_{ri}f^{rk}+p_{rr}^{k}f^{i}+p_{rr}f^{ik}+\delta_{rk}f^{ir}+x_rf^{irk})\big)\\
 &-q_i\big(p_{kk}f+x_kf^{k}\big)+\sum_r p_{ri}\big(p_{kk}^{r}f+p_{kk}f^{r}+\delta_{kr}f^{k}+x_kf^{kr}\big)\\
 &+\sum_r p_{rr}\big(p_{kk}^if+p_{kk}f^i+\delta_{ki}f^{k}+x_kf^{ki}\big)\\
 &+\sum_r x_r\big(p_{kk}^{ri}f+p_{kk}^{i}f^r+p_{kk}^rf^i+p_{kk}f^{ir}+\delta_{ki}f^{kr}+\delta_{kr}f^{ki}+x_kf^{irk}\big)\\
 &+\delta_{ki}\big(q_if-\sum_r(p_{ri}f^r+p_{rr}f^i+x_rf^{ir})\big)\\
=& \; \Big(x_kq_i^k+\delta_{ki}q_i +\sum_{r=1}^n (p_{kk}^rp_{ri}+x_rp_{kk}^{ir}+p_{kk}^{i}p_{rr}) \Big)f\\
 &+\Big(\sum_r (x_rp_{kk}^r-x_kp_{rr}^k)\Big)f^i+\sum_{r}\Big(x_rp_{kk}^i-x_kp_{ri}^k-\delta_{ki}p_{ri}\Big)f^r+p_{ki}f^k\\
 =& \Big(x_kq_i^k+\delta_{ki}q_i +\sum_{r=1}^n (p_{kk}^rp_{ri}+x_rp_{kk}^{ir}+p_{kk}^{i}p_{rr}) \Big)f
\end{align*}
where the last equality followed by using equations~\eqref{pii} and~\eqref{pij}. Since the above equality holds for all $f$ we have 
\begin{equation}
x_kq_i^k+\delta_{ki}q_i +\sum_{r=1}^n (p_{kk}^rp_{ri}+x_rp_{kk}^{ir}+p_{kk}^{i}p_{rr})=0 \label{relhq}
\end{equation}
Taking $k=i$ we obtain
\begin{align*}
x_iq_i^i+q_i &+\sum_{r=1}^n (p_{ii}^rp_{ri}+x_rp_{ii}^{ir}+p_{ii}^{i}p_{rr})=0\\
&\Leftrightarrow \dd{x_i}(x_iq_i)=-\sum_{r=1}^n (p_{ii}^rp_{ri}+x_rp_{ii}^{ir}+p_{ii}^{i}p_{rr})\\
&\Leftrightarrow q_i=-\frac{1}{x_i} \int \sum_{r=1}^n (p_{ii}^rp_{ri}+x_rp_{ii}^{ir}+p_{ii}^{i}p_{rr})dx_i\\
& \quad = -\frac{1}{x_i} d_i'\big(x_i\sum_{r=1}^n (p_{ii}^rp_{ri}+x_rp_{ii}^{ir}+p_{ii}^{i}p_{rr})\big).
\end{align*}
Note that since $q_i$ is a polynomial the integral above is well-defined since it needs to be divisible by $x_i$. 
\end{proof}

\begin{thm}
\label{mainthm}
Fix a polynomial $p\in \pol$ and for $1\leq i,j \leq n$ define polynomials $p_{ij}:=x_i\frac{\partial p}{\partial x_j}+\delta_{ij}p(0)/n$ and $q_i:=-\frac{1}{x_i} \int \sum_{r=1}^n (p_{ii}^rp_{ri}+x_rp_{ii}^{ir}+p_{ii}^{i}p_{rr})dx_i$.

Then the following action defines an $\fr{sl}_{n+1}$-module structure on the space $M(p)=\pol$:
\begin{align*}
e_{i,n+1}\cdot f&=x_if\\
h_i\cdot f&=fp_{ii} + x_if^i\\
e_{ij}\cdot f&=fp_{ij}+x_if^j\\
e_{n+1,i}\cdot f&=q_if-\sum_r(p_{ri}f^r+p_{rr}f^i+x_rf^{ir})
\end{align*}
where $h_i=e_{ii}-\frac{1}{n+1}\sum_{j=1}^{n+1}e_{jj}$.

Moreover, any $\fr{sl}_{n+1}$-module $M$ for which $\Res_{U(\fr{n})}^{U(\fr{sl}_{n+1})}M$ is free of rank $1$ is isomorphic to $M(p)$ for a unique $p\in \pol$.
\end{thm}
\begin{proof}
We first verify that the definition in the theorem indeed gives an $\fr{sl}_{n+1}$-module structure.
Lemma~\ref{Ucor} guarantees that for any $y\in \fr{sl}_{n+1}$ we have $[y,x_k]\cdot f = y\cdot x_k \cdot f - x_k\cdot y \cdot f$.

By equations~\eqref{pii},~\eqref{pij}, \eqref{pij3}, and~\eqref{pij2} we see that the relations
\begin{align*}
[h_i,h_j]\cdot f&=h_i\cdot h_j \cdot f - h_j\cdot h_i \cdot f\\
[e_{ij},h_k]\cdot f&=e_{ij}\cdot h_k \cdot f - h_k\cdot e_{ij} \cdot f\\
[e_{ij},e_{kl}]\cdot f&=e_{ij}\cdot e_{kl} \cdot f - e_{kl}\cdot e_{ij} \cdot f\\
\end{align*}
holds for all $f$ if and only if
\[x_ip_{kl}^j-x_kp_{ij}^l=\delta_{kj}p_{il}-\delta_{il}p_{kj}\]
for all $1\leq i,j,k,l \leq n$. We verify these relations for $p_{ij}=x_ip^j+\delta_{ij}c$:

\begin{align*}
x_ip_{kl}^j-x_kp_{ij}^l=&x_i(x_kp^l)^j-x_k(x_ip^j)^l\\
	&=x_ix_kp^{lj}+\delta_{kj}x_ip^l-x_kx_ip^{jl}-\delta_{il}x_kp^j\\
	&=\delta_{kj}x_ip^l-\delta_{il}x_kp^j\\
	&= \delta_{kj}(p_{il}-c\delta_{il})-\delta_{il}(p_{kj}-c\delta_{kj})\\
	&= \delta_{kj}p_{il}-\delta_{il}p_{kj}.
\end{align*}

It remains only to show that $[e_{n+1,i},y]\cdot f=e_{n+1,i}\cdot y \cdot f - y\cdot e_{n+1,i} \cdot f$ for each $y\in \fr{sl}_{n+1}$.
We verify only the relation for $y=h_k$ here, the remaining relations are similar.

By Lemma~\ref{qlemma}, the condition $[h_k,e_{n+1,i}]\cdot f=h_k \cdot e_{n+1,i} \cdot f-e_{n+1,i} \cdot h_k \cdot f$ for all $f$ is equivalent to 
\begin{equation}
x_kq_i^k+\delta_{ki}q_i +\sum_{r=1}^n (p_{kk}^rp_{ri}+x_rp_{kk}^{ir}+p_{kk}^{i}p_{rr})=0. \label{relhq2}
\end{equation}
For $k=i$, this equation holds by the definition of $q_i$, so assume that $k\neq i$. Then we have 
\begin{align*}
 &x_kq_i^k+\delta_{ki}q_i +\sum_{r=1}^n (p_{kk}^rp_{ri}+x_rp_{kk}^{ir}+p_{kk}^{i}p_{rr})=0\\
 &\Leftrightarrow d_k(d_i(-x_iq_i)) =d_i\big(x_i\sum_{r=1}^n (p_{kk}^rp_{ri}+x_rp_{kk}^{ir}+p_{kk}^{i}p_{rr})\big)\\
 &\Leftrightarrow x_i\sum_{r=1}^n d_k(p_{ii}^rp_{ri}+x_rp_{ii}^{ir}+p_{ii}^{i}p_{rr})) =d_i\big(x_i\sum_{r=1}^n (p_{kk}^rp_{ri}+x_rp_{kk}^{ir}+p_{kk}^{i}p_{rr})\big)\\
 &\Leftrightarrow \sum_{r=1}^n x_k\dd{x_k}(p_{ii}^rp_{ri}+x_rp_{ii}^{ir}+p_{ii}^{i}p_{rr})) =\dd{x_i}\big(x_i\sum_{r=1}^n (p_{kk}^rp_{ri}+x_rp_{kk}^{ir}+p_{kk}^{i}p_{rr})\big)
\end{align*}
Substituting $p_{ij}=x_ip^j+\delta_{ij}p(0)/n$, the verification of the above equality reduces to a simple but long calculation which we omit here.

Finally we note that the conditions
\[[e_{n+1,i},e_{n+1,j}]\cdot f=e_{n+1,i}\cdot e_{n+1,j} \cdot f - e_{n+1,j}\cdot e_{n+1,i} \cdot f\]
\[[e_{n+1,i},e_{jk}]\cdot f=e_{n+1,i}\cdot e_{jk} \cdot f - e_{jk}\cdot e_{n+1,i} \cdot f\]
reduces to the two conditions
\[\sum_r(p_{ri}q_k^r+p_{rr}q_k^i+x_rq_k^{ir})=\sum_r(p_{rk}q_i^r+p_{rr}q_i^k+x_rq_i^{kr})\]
\[x_jq_i^k+\sum_r(p_{ri}p_{jk}^r+p_{rr}p_{jk}^i+x_rp_{jk}^{ir})=0\]
which can be verified similarly.

Next we adress the uniqueness claim of the theorem. Suppose that we are given an $\fr{sl}_{n+1}$-module structure on $M=\pol$ such that
$e_{i,n+1}\cdot f = x_if$ for $1 \leq i \leq n$. Define $p_{ij}:=e_{i,j}\cdot 1$ and $q_i:=e_{n+1,i}\cdot 1$ as before. By Lemma~\ref{act}, $M$ is determined by these polynomials up to isomorphism. We need to show that $M\simeq M(p)$ for some $p\in \pol$.

Equation~\eqref{pii} says that $x_ip_{jj}^i=x_jp_{ii}^j$ for all $1\leq i,j \leq n$. Summing over $j$ we obtain
$x_i\ol{p}^i=d(p_{ii})$ and applying $d'$ we have
\[d'(x_i\ol{p}^i)=d'\circ d(p_{ii})=p_{ii}-p_{ii}(0,\ldots,0)\]
 Thus each $p_{ii}$ is determined up to addition of a constant by $\ol{p}$.
We conclude that \[p_{ii}=d'(x_i\ol{p}^i)+c_i\]
for some constants $c_i$.

Next recall that by Equation~\ref{pij3} we had $x_ip_{kl}^j-x_kp_{ij}^l=\delta_{kj}p_{il}-\delta_{il}p_{kj}$.
Taking $k=i$ and $j=l$ here we obtain
\begin{equation}
\label{pij2}
x_ip_{ji}^j-x_jp_{ij}^i=p_{ii}-p_{jj}.
\end{equation}
Substituting $p_{ii}=d'(x_i\ol{p}^i)+c_i$ in~\eqref{pij2} and considering the constant term it follows that $c_i=c_j$ for all $i,j$ and
\[p_{ii}=d'(x_i\ol{p}^i)+c=x_i\tfrac{\partial}{\partial x_i}d'(\ol{p})+c.\]
Define $p:=d'(\ol{p})$. Then $p_{ii}=x_ip^i+c$, and $\ol{p}(0)=\sum_{i=1}^n p_{ii}(0) = nc$, so 
\[c=\frac{\ol{p}(0)}{n}=\frac{d'(\ol{p})(0)}{n}=\frac{p(0)}{n}.\] 

Finally, recall that by~\eqref{pij} we had $x_kp_{ij}^k-x_ip_{kk}^j=(\delta_{ki}-\delta_{kj})p_{ij}$.
Summing over $1\leq k \leq n$ in we get
\[d(p_{ij})-x_i\ol{p}^j=(\sum_k\delta_{ki}-\sum_k\delta_{kj})p_{ij}=(1-1)p_{ij}=0.\]
Thus $p_{ij}-p_{ij}(0) = d'\circ d (p_{ij})=d'(x_i\ol{p}^j)$.
Taking $k=i$ in \eqref{pij} we get $x_ip_{ij}^i-x_ip_{ii}^j=p_{ij}$ which shows that $p_{ij}(0)=0$, and we conclude that
\[p_{ij}=d'(x_i\ol{p}^j)=x_ip^j.\]

Thus we have shown that the $p_{ij}$ defined above are the same as those in $M(p)$. Finally, by Lemma~\ref{qlemma} each $q_i$ is uniquely determined by the $p_{ij}$. We have therefore proved that $M\simeq M(p)$, which shows that the $\fr{sl}_{n+1}$ modules of form $M(p)$ exhaust the set of modules which are free when restricted to $U(\fr{n})$.
\end{proof}

\subsection{Submodule structure}
\label{sec:sub}
\begin{prop}
We have $M(p) \simeq M(\tilde{p})$ if and only if $p=\tilde{p}$.
\end{prop}
\begin{proof}
Let $\varphi: M(p) \ra M(\tilde{p})$ be an isomorphism. Since $\varphi(f)=\varphi(f\cdot 1)= f\cdot \varphi(1)$,  $\varphi(1)$ is a nonzero constant. For $i\neq j$ we have $\varphi(e_{ij}\cdot 1)=e_{ij}\cdot \varphi(1)$, which implies $x_ip^j = x_i\tilde{p}^j$ and $p^j=\tilde{p}^j$. Since this holds for each $j$, $p-\tilde{p}$ is a constant. Finally, $\varphi(h_j\cdot 1)=h_j\cdot \varphi(1)$ implies 
$x_jp^j+p(0)/n=x_j\tilde{p}^j+\tilde{p}(0)/n$ and in turn $p(0)=\tilde{p}(0)$. Thus $p=\tilde{p}$.
\end{proof}

\begin{lemma}
Every submodule $N$ of $M(p)$ is a homogeneous ideal of form \[N=\{f\in M(p) | \deg(f)\geq k\}\] for some $k\in \bb{N}$.
\end{lemma}
\begin{proof}
Let $N$ be a submodule of $M(p)$. Freeness of $M(p)$ over $\pol$ implies that any submodule is an ideal.
Note that $D_i:=(h_i-x_ip^i-p(0)/n)\in U(\fr{sl}_{n+1})$ acts as the $i$-degree operator on $M(p)$, and $n-D_i$ acts on $f$ by killing all terms of $f$ which has $i$-degree $n$. Thus any $f\in N$ can be reduced to each of its monomial terms. Next, let $f \in N$.
Then for $i\neq j$, $(e_{ij}-x_ip^j)\cdot f=x_if^j$ which shows that $f$ can be mapped to any polynomial of the same degree by a product of such elements in $U(\fr{sl}_{n+1})$. Thus $N \supset \{g \:|\: d(g)=d(f)\}$, and since $\K[x]$ acts freely on $N$, we have $N \supset \{g \:|\: d(g)\geq d(f)\}$. Therefore if $f\in N$ is taken to have minimal degree we see that $N$ has the form stated in the lemma.
\end{proof}

\begin{thm}$\:$\\
$M(p)$ is a simple $\fr{sl}_{n+1}$-module if and only if $k:=-\frac{n+1}{n}p(0,\ldots, 0) \not\in \bb{N}_+$. Otherwise, if $k\in \bb{N}_+$, the module $M(p)$ has length $2$ and the top of $M$ is a simple highest weight module.
\end{thm}
\begin{proof}
Let $N\subset M(p)$ be a nonzero proper submodule, and $f\in N$ be a monomial with $\deg f$ minimal.
Then \[\big(q_i-e_{n+1,i}\big)\cdot f=\sum_r(x_rp^i+\delta_{ri}p(0)/n)f^r+(x_rp^r+p(0)/n)f^i+x_rf^{ir}.\]
Subtracting terms of degree $\geq \deg(f)$ we obtain
\[(p(0)/n)f^i+n(p(0)/n)f^i+d(f^i)=(\tfrac{n+1}{n}p(0)+\deg(f^i))f^i\in N.\]
But by the minimality of $\deg(f)$, the coefficient of $f^i$ must be zero and we have $\deg(f^i)=-\frac{n+1}{n}p(0)$. This is a contradiction if $k\not\in \bb{N}_+$, so $M(p)$ is simple in this case. Conversely, assume $k\in \bb{N}_+$ and let
\[W_k:= \mathrm{span}\{x_1^{a_1}\cdots x_n^{a_n} \in M(p) | \sum a_{i} \geq k\}.\]
Then $W_k$ is invariant under the action of each $e_{n+1,i}$ by the above calculation. $W_k$ is also invariant under the remaining basis elements of $\fr{sl}_{n+1}$ because $\deg(x_if^j)=\deg(f)$ for all $1\leq i,j \leq n$. Thus $W_k$ is the unique submodule of $M(p)$. This means that $M(p)/W_k$ is simple finite-dimensional module, and therefore a weight module. We note that the image of $x_1^k$ in the quotient is a highest weight vector since it is annihilated by all $e_{ij}$ for $i<j$. Since 
\[h_i \cdot x_1^k = \delta_{ik}(k+p(0)/n)x_1^k=\delta_{ik}\tfrac{n+2}{n}p(0)x_1^k,\]
we have $M(p)/W_k \simeq L(\lambda)$ where $\lambda \in \fr{h}^*$ is given by $\lambda(h_i)=\delta_{ik}\frac{n+2}{n}p(0)$. We note that $\dim M(p)/W_k = {k+n-2 \choose k-1}$, the dimension of the space of polynomials of degree less than $k$.
\end{proof}


\begin{thebibliography}{9}
\bibitem[B]{B} R.~Block. \emph{The irreducible representations of the Lie algebra $\mathfrak{sl}(2)$ 
and of the Weyl algebra}. Adv. in Math. {\bf 139} (1981), no. 1, 69--110.
\bibitem[BGG]{BGG}
I.~N.~Bernstein, I.~M.~Gelfand, S.~I.~Gelfand; \emph{A certain category of $\mathfrak{g}$-modules}.
Funkcional. Anal. i Prilozen. {\bf 10}, 1--8. 
\bibitem[BL]{BL} D.~Britten, F.~Lemire; \emph{Irreducible representations of $A_n$ with a $1$-dimensional weight space}.
Trans. Amer. Math. Soc. {\bf 273} (1982), 509--540.
\bibitem[BM]{BM} P.~Batra, V.~Mazorchuk; \emph{Blocks and modules for Whittaker pairs}. 
J. Pure Appl. Algebra {\bf 215} (2011), no. 7, 1552--1568. 
\bibitem[Ca]{Ca} E.~Cartan; \emph{Les groupes projectifs qui ne laissent invariante aucune multiplicit{\'e} 
planet.} Bull. Soc. Math. France vol. {\bf 41} (1913) pp. 53--96.
\bibitem[CC]{CC}
Q.~Chen, Y.~Cai. \emph{Modules over algebras related to the Virasoro algebra}
International Journal of Mathematics {\bf 26}, No. 09, 1550070 (2015).
\bibitem[CG]{CG}
H.~Chen, X.~Guo. \emph{Non-weight modules over the Heisenberg–Virasoro algebra and the $W$ algebra $W(2,2)$}.
Journal of Algebra and Its Applications Vol. {\bf 16}, No. 05, 1750097 (2017).
\bibitem[CLNZ]{CLNZ}
Y.~Cai, G.~Liu, J.~Nilsson, K.~Zhao.\emph{Generalized Verma modules over $\fr{sl}_{n+2}$ induced from $U(\fr{h}_n)$-free $\fr{sl}_{n+1}$-modules}.
Journal of Algebra {\bf 502} (2018), 146--162.
\bibitem[CTZ]{CTZ}
Y.~Cai, H.~Tan, K.~Zhao. \emph{Module structure on $U(\fr{h})$ for Kac-Moody algebras.}
Preprint arXiv:1606.01891, 2016.
\bibitem[CZ]{CZ}
Y.~Cai, K.~Zhao. \emph{Module structure on $U(\fr{h})$ for basic Lie superalgebras.}
Toyama Math. J. {\bf 37} (2015), 55--72.
\bibitem[Di]{Di}
J.~Dixmier; \emph{Enveloping Algebras.} American Mathematical Society, 1977.
\bibitem[DFO]{DFO} Yu.~Drozd, S.~Ovsienko, V.~Futorny; \emph{On Gelfand-Zetlin modules}. Proceedings of the 
Winter School on Geometry and Physics (Srni, 1990). Rend. Circ. Mat. Palermo (2) Suppl. No. {\bf 26} (1991), 143--147.
\bibitem[Fe]{Fe} S.~Fernando; Lie algebra modules with finite dimensional weight spaces, I. 
Trans. Amer. Mas. Soc., {\bf 322} (1990), 757--781.
\bibitem[FOS]{FOS} V.~Futorny, S.~Ovsienko, M.~Saor\'{i}n; \emph{Torsion theories induced from commutative subalgebras}.
J. Pure Appl. Algebra {\bf 215} (2011), no. 12, 2937--2948. 
\bibitem[Fu]{Fu} V.~Futorny; \emph{Weight representations of semisimple finite dimensional Lie algebras}. 
Ph. D. Thesis, Kiev University, 1987. 
\bibitem[HCS]{HCS}J.~Han, Q.~Chen, Y.~Su. \emph{modules over the algebra $Vir(a,b)$}.
Linear Algebra and its Applications, Volume {\bf 515} (2017) 11--23.
\bibitem[Hu]{Hu}
J.~E.~Humphreys;
\emph{Representations of Semisimple Lie Algebras in the BGG Category $\mathcal{O}$.}
American Mathematical Society, 2008.
\bibitem[Kob]{Kob}
T.~Kobayashi; \emph{Multiplicity-free Theorems of the Restrictions of Unitary Highest Weight Modules with respect to Reductive Symmetric Pairs.}
Representation Theory and Automorphic Forms, 45--109. Progr. Math., {\bf 255}. Birkhäuser Boston, Boston, MA 2008.
\bibitem[Kos]{Kos}
B.~Kostant;  \emph{On Whittaker vectors and representation theory}. 
Invent. Math. {\bf 48} 1978, no. 2, 101--184.
\bibitem[LZ]{LZ}
G.~Liu, Y.~Zhao. \emph{Generalized polynomial modules over the Virasoro algebra.}
Proc. Amer. Math. Soc. {\bf 144} (2016), 5103--5112. 
\bibitem[Mat]{Mat}
O.~Mathieu; \emph{Classification of irreducible weight modules},
Ann. Inst. Fourier (Grenoble) {\bf 50} (2000), no. 2, 537--592. 
\bibitem[Maz1]{Maz1} 
V.~Mazorchuk; \emph{Lectures on $\mathfrak{sl}_2(\mathbb{C})$-modules}. Imperial College Press,
London, 2010.
\bibitem[MP]{MP}
F.~J.~P.~Martín, C.~T.~Prieto \emph{Construction of simple non-weight -modules of arbitrary rank.}
Journal of Algebra {\bf 472} (2017), 172--194.
\bibitem[N1]{N1}
J.~Nilsson. \emph{Simple $\mathfrak{sl}_{n+1}$-module structures on $U(\fr{h})$.}
J.~Algebra {\bf 424} (2015), 294--329.
\bibitem[N2]{N2}
J.~Nilsson. \emph{$U(\fr{h})$-free modules and coherent families.} Journal of Pure and Applied Algebra (220) 1475--1488.
\bibitem[N3]{N3}
J.~Nilsson. \emph{A New family of simple $\fr{gl}_{2n}(\bb{C})$-modules.}
Pacific Journal of Mathematics {\bf 283} (2016), No. 1, 1--19.
\bibitem[TZ1]{TZ1}
H.~Tan, K.~Zhao. \emph{Irreducible Modules Over Witt Algebras $\mathcal {W}_{n}$ and Over $\mathfrak{sl}_{n+1}(\mathbb{C})$}.
Algebr Represent Theor (2018) 21:787--806.
\bibitem[TZ2]{TZ2}
H.~Tan, K.~Zhao. \emph{$\ca{W}_{n}^{+}$- and $\ca{W}_{n}$-module structures on $U(\mathfrak{h}_{n})$}.
J.~Algebra {\bf 424} (2015), 357--375.
\end{thebibliography}
\end{document}